\documentclass[pdflatex,sn-mathphys-num]{sn-jnl}% Math and Physical Sciences Numbered Reference Style
\usepackage{mathrsfs}%
\usepackage[title]{appendix}%
\usepackage{xcolor}%
\usepackage{textcomp}%
\usepackage{manyfoot}%
\usepackage{booktabs}%
\usepackage{algorithm}%
\usepackage{algorithmicx}%
\usepackage{algpseudocode}%
\usepackage{listings}%
\usepackage{amsmath,amsthm,amssymb,mathrsfs,bm}

\usepackage{graphicx} % figures removed but package harmless
\usepackage{geometry}
\usepackage{hyperref}
\usepackage{stmaryrd}
\usepackage{float}
\usepackage{multirow}
\usepackage{tikz} % kept though not used for images

\theoremstyle{thmstyleone}%
\newtheorem{theorem}{Theorem}
\theoremstyle{thmstyletwo}%
\newtheorem{remark}{Remark}%
\theoremstyle{thmstylethree}%

\newcommand{\bigslant}[2]{{\raisebox{.2em}{$#1$}\left/\raisebox{-.2em}{$#2$}\right.}}
\newcommand{\R}{\mathbb{R}}
\newcommand{\RP}{\mathbb{R}\mathbb{P}}
\newcommand{\C}{\mathbb{C}}
\newcommand{\p}{{\mathbb{P}}}
\newcommand{\K}{{\mathbb{K}}}
\newcommand\blfootnote[1]{
    \begingroup
    \renewcommand\thefootnote{}\footnote{#1}
    \addtocounter{footnote}{-1}
    \endgroup
}
\begin{document}

\title{Introduction to non-Abelian Patchworking}
\author{Turgay Akyar and Mikhail Shkolnikov}

%\author*[1,2]{\fnm{Mikhail} \sur{Shkolnikov}}\email{iauthor@gmail.com}

%\author[2,3]{\fnm{Turgay} \sur{Akyar}}\email{iiauthor@gmail.com}


%\affil*[1]{\orgdiv{Department}, \orgname{Organization}, \orgaddress{\street{Street}, \city{City}, \postcode{100190}, \state{State}, \country{Country}}}

%\affil[2]{\orgdiv{Department}, \orgname{Organization}, \orgaddress{\street{Street}, \city{City}, \postcode{10587}, \state{State}, \country{Country}}}

%\affil[3]{\orgdiv{Department}, \orgname{Organization}, \orgaddress{\street{Street}, \city{City}, \postcode{610101}, \state{State}, \country{Country}}}


%%\abstract{The note introduces a novel concept of Non-Abelian patchworking which is a logical real algebraic geometry spin-off of the developing theory of Non-Abelian (previously referred to as Non-Commutative) Amoebas and their tropical limits. The main focus is made on proposing a new method for constructing topological types of surfaces in real projective 3-space, seen as the projectivization of a four dimensional vector space of two-by-two matrices.}

\abstract{The note introduces a novel concept of non-Abelian patchworking arising as real locus of non-Abelian complex-phase tropical hypersurfaces, the theory of which is now developed enough to allow the proposed spin-off. Although, non-Abelian Tropical Geometry makes sense for an arbitrary reductive complex group, the state of the art is that of full understanding of tropicalizations of surfaces within three dimensional groups $PGL_2(\mathbb{C})$ and $SL_2(\mathbb{C}),$ which are closely related via the two-fold covering. We stress our point, that this is an announcement of a framework, taking care of explaining explicitly the input, which is more geometric and less combinatorial than in the original Viro's method, to construct possible types of real algebraic surfaces in the real projective 3-space, and verify that it reproduces all the existing isotopy types of surfaces up to degree three. We obtain two general theorems concerning the topology of primitive PGL2 surfaces, observing in particular that they may have different Euler charteristic for a fixed degree greater than one, not necessarily equal to the signature of the corresponding complex surface, which would be the case for primitive combinatorial patchworking due to a result of Itenberg.}

\keywords{patchworking, PGL2, tropical geometry, real algebraic surfaces.}



\maketitle

\section{Introduction}

The study of the topology of real algebraic varieties is a very old and complicated subject dating back to the antiquity, with the first non-trivial instance being the classification of conic curves in the plane. As is universally known, the name “conic” refers to conic sections, i.e. the intersections of a plane with a cone given by $z^2= x^2+y^2$ for $(x,y,z)\in\mathbb{R}^3$, which is an example of a singular affine quadric surface. Thus, two terms “conic” and “quadratic curve” are almost synonyms: every (affine) conic, which can be an ellipse, a hyperbola, or a parabola, in the non-singular case, and a union of two intersecting lines, or a single point in the singular case, is a quadratic curve. One missing case is an ``empty conic'' given for instance by $x^2+y^2=-1,$ and that of a pair of parallel lines in the real affine plane $\mathbb{R}^2$. 
\blfootnote{The order of the authors' names is alphabetical.}

Nevertheless, the two lines can be thought to be intersecting at infinity. This, in part, motivates compactification of the affine plane $\R^2$ to the projective plane $\RP^2$ by adding all such ideal points of intersection at infinity, ensuring the property that every two lines intersect. From this perspective, all smooth non-empty conics are now same up to projective equivalence given by the action of an  eight-dimensional group $PGL_3(\R)$, a projectivization, i.e. the quotient by rescaling, of the group $GL_3(\R)$ of $3\times 3$ invertible matrices. Such matrices are the linear automorphisms of $\R ^3$ seen as a vector space, and $\RP^3$ is realized as projectivization of $\R^3$. The original $\R^2$ may be viewed as the image of a plane $z=1$ under the projection $\R^3\longrightarrow\frac{\R^3\setminus \{0\}}{\R\setminus\{0\}}=\RP^2.$ Practically, this means that every polynomial equation $\displaystyle\sum_{k+e\leq d}^{} a_{k,e}x^ky^e=0$ of degree $d$ can be homogenized to$$\displaystyle\sum_{k+e\leq d} a_{k,e}x^ky^ez^{d-k-e}=0,$$ where every monomial now is of total degree $d$.   The homogeneous equation of degree $d>1$ necessarily defines a singular surface (with the singularity at the origin) which is invariant under the rescaling, and they can be projectivized to a curve in $\RP^2$. This curve is called singular if there exist $p\in\R^3\setminus\{0\}$ such that $\partial_xf(p)=\partial_yf(p)=\partial_zf(p)=0$,
where $f$ is the homogenized defining polynomial for the projective curve under consideration. Note that by Euler's identity
$d\,f=\partial_xf+\partial_yf+\partial_zf$, the point $p$ projects to the (singular) point of the curve in $\RP^2$. 
As it was mentioned, the smooth projective conics are all equivalent to each other, and what it makes look them different is their relative positions with respect to the line at infinity defined by the equation $z = 0$; hyperbolas intersect this line twice, parabolas are tangent to this line, intersecting it at a single point, and ellipses have no points at infinity. Topologically, a smooth projective conic is a circle, whose complement in $\RP^2$ consists of a disjoint union of an open disc (the interior), and of an open Möbius band (the exterior). Smooth degree 3 projective curves, also known as cubics, admit two different topological types consisting of either a single connected component, which is isotopic to a line, i.e. being its continuous  deformation, or of two components, being isotopic to a disjoint union of a line and an ``oval'', with the latter term referring to an embedded circle in $\mathbb{R}P^2$ isotopic to a smooth non-empty projective conic. To construct such curves, simply take union of a smooth (either empty or non-empty) conic and a line.

Quartic curves may have at most four connected components, each of which is an oval. Moreover, if one such component belongs to the interior of another, the quartic has no other components. To see this, assume the contrary, and draw a line through a point in the intersection of interiors of the two nested ovals and a point in the interior of another oval. This line would intersect each of these three ovals at at least two points resulting in six or more intersections, which is impossible by a simple instance of the Bézout's theorem --since the restriction of the defining degree four polynomial to the line (which in our smooth case is not contained in the quartic) can have at most four roots.

Similarly, one can easily work out the degree five case. Each smooth projective quintic is isotopic to a union of a line (with corresponding component traditionally called a “pseudo-line”) and at most six ovals. Moreover, if there is a pair of nested ovals, the total number of components must be 3, which follows from the same B\'ezout-type prohibition argument. 

In general, the total number of components of a degree $d$ projective curve is majorated by $$\dfrac{(d-1)(d-2)}{2}+1,$$ which is Harnack's inequality, who in addition proved that this bound is sharp in every degree. The above bound is elegantly explained by Klein's argument, relating, more generally, the genus (i.e. the number of handles) of the complexification (i.e. the complex set of solutions of the same equation, which is a topological surface) to the number of components of the real locus by computing the Euler characteristic of the quotient of the complexification with respect to the involution given by coordinate-wise conjugation. For example, both degree $1$ and $2$ define a complex curve of genus $0$, i.e. a topological $2$-sphere, which corresponds to a single real component. In degree $3$ case, smooth complex projective cubic topologically is a torus, a surface of genus $1$, thus we may have at most $2$ components in the real locus. In degree $4$, the smooth complex solution set is a surface of genus $3$, which gives at most $4$ components, and in degree $5$, the genus is $6$, so we have at most $7$ components, including the pseudo-line.

The first part of the Hilbert's XVI problem was concerned with possible relative positions of ovals of a sextic, i.e. degree $6,$ curve with the maximal number of components, which is $11$. Harnack's general geometric construction provided a configuration with one nest and 9 isolated ovals. In addition, Hilbert gave a similar construction in even degree case, which after being specialized to sextics, gave an ``opposite'' configuration with one isolated oval and an oval encompassing nine other ones. He further conjectured that no other topological configuration exists. After a long struggle, a third type was constructed by Gudkov decades later, which is somewhat in between Hilbert's and Harnack's configurations, with one oval encompassing five ovals, and the five remaining ones being outside. This fascinating story is beautifully described in an essay of Oleg Viro \cite{Viro08}. Nowadays, the problem of relative position of ovals of a smooth projective curve is known up to degree $7$, using various sophisticated prohibition and construction methods. Currently, the most potent and universal construction is Viro's combinatorial patchworking. 

This method is indeed quite combinatorial and seemingly unrelated to the continuous nature of the problem at hand. It operates at the level of the Newton polygon of the curve, which in the projective plane case is a triangle with vertices $(0, 0), (d, 0)$ and $(0, d)$, as well as its lattice triangulation, where every node has integer coordinates, and an arbitrary distribution of signs at them. Using this datum, one constructs a piecewise linear curve in the union of four reflections (with respect to the coordinate lines) of the Newton polygon, which form together a centrally symmetric diamond shape. After identifying the opposite points on its boundary, one gets $\RP^2$ topologically, and a configuration of embedded circles. The main theorem about combinatorial patchworking is that this configuration is realizable by some real algebraic curve of degree $d,$ provided that the triangulation satisfies a technical ``convexity'' assumption, which is actually not known to be necessary. In modern terms, Viro's patchworking is the topological type of a tropical curve enhanced with a real phase structure, and the theorem is that the topological type of the real curve near the tropical limit is restored in this way. 

The generality of the Viro's method is much more than that of curves in $\RP^2$. On the one hand, one can work with general smooth toric surfaces. On the other hand, we may increase the dimension arbitrarily and be concerned with hypersurfaces in higher-dimensional toric varieties, and even more generally, with tropical complete intersections therein, see \cite{RRS25} for the most recent development of patchworks in higher codimension via real phase structures.

An important conceptual remark here is that one could also have a similar
topology restoration result in the complex case, which was conjectured by Viro. Namely, theorems of Kerr and Zharkov \cite{KerrZharkov18}, or Kim and Nisse \cite{KimNisse21}, state that complex phase enhancement of a smooth tropical hypersurface recovers the topology of the corresponding complex hypersurface. The notion of tropical smoothness, at the level of the dual subdivision of a Newton polyhedron reads as requiring the subdivision to be a triangulation with every full dimensional sub-simplex is primitive, having the minimal possible volume, i.e. $\frac{1}{2}$ in the two dimensional use, and $\frac{1}{6}$ in 3 dimensions. This condition ensures that the variety near the tropical limit is smooth manifold.

Patchworking with this tropical-smoothness condition is referred to as ``primitive''. Primitive patchworking, being still quite powerful, has nevertheless a strong limitation for constructing topological types of even degree hypersurfaces. For instance, Itenberg \cite{Itenberg97} observed that the Euler characteristic of the primitive patchworking in $\RP^3$ depends only on the degree, at first glance, rather coincidentally being equal to the signature of the smooth complex surface of corresponding degree in $\C\p^3$ seen as a topological four manifold. Similar result holds in higher dimensions, as it was established by Bertrand \cite{Bertrand10} using the same technique of counting numbers of simplices of various dimensions. Nowadays, the conceptual explanation is provided by Renaudineau-Shaw spectral sequence \cite{RenaudineauShaw23} realizing the homology of the real phase tropical variety as the limit of the (bi-graded) tropical homology \cite{IKMZ19} of the smooth tropical variety.

The status of the adaptation of Hilbert's question for higher dimensional varieties is rather unsatisfactory presently. On the one hand, Harnack's inequality appropriately generalizes to the Smith-Thom inequality, which  also can be achieved by a spectral sequence argument, giving the sharp upper bound on the sum of all Betti members of the real locus in terms of the corresponding quantity of the complexification. The individual Betti members, such as the $0$-th one counting the number of connected components, also admit some reasonable bounds in terms of Hodge numbers of complexification, but they are not known to be sharp. In particular, the known best upper bound on the number of connected components of a real quintic surface in $\RP^3$ is $25$, but the highest number of component construction was done by Orevkov a quarter of a century ago giving an example with $23$ components \cite{Orevkov01}. In other words, we still don't know if there exist projective quintics with $24$ on $25$ connected components. Similarly, for the first Betti number, i.e. the number of independent loops, there have been some progress (see for instance \cite{Renaudineau15}), but the question is far from being resolved.

The situation is even more complicated if one wants to know the relative position of these components especially in higher codimension, such as for real algebraic knots, see for instance \cite{B11,MO19}, understanding of isotopy types is far from being complete. New prohibitions and constructions are required. What we propose in this note, is a conceptually new construction method inspired by Viro's original idea, but put
in a different framework of non-Abelian tropicalization \cite{MS22} and its phase version \cite{SP24}, with its novel rendition in the real-phase case. Applied to the group $PGL_2(\C)$, this gives a conjectural method of constructing surfaces in $\RP^3$, which is rather different in flavor to the classical one, being much less combinatorial, and reducing the three-dimensional geometry to the problem of relative position of pairs of smooth curves on ellipsoids and hyperboloids.

\section{$PSL_2(\C)$-tropical surfaces}
This section is dedicated to outlining the construction of $PSL_2$ phase tropical surface subject to smoothness condition and in the complex coefficient case. The corresponding real counterpart will arise by asking the data of the construction to be invariant with respect to some proffered anti-holomorphic involution (a.k.a. “real structure”) on $PSL_2(\mathbb{C})$ and by taking the corresponding fixed point locus.

One way to approach the notion of $PSL_2$-tropical surface is by looking at the closure of the image of a surface in $PSL_2(\mathbb{K})$ under the map $$\operatorname{VAL}: PSL_2(\K)\rightarrow PSL_2(\C),$$ where $\K$ is an algebraically closed non-Archimedean field with the residue field $\C$. The above map may be thought of as “phase tropicalization of a family of points”, the explicit formula, utilizing the notion of matrix adjugate, is: $$\operatorname{VAL}[t^\alpha B+o(t^\alpha)]_{\K^*}=[e^\alpha B+e^{-\alpha}(B^*)^{adj}]_{\C^*}.$$
Description of the structure of the images of surfaces under this map was conjectured in \cite{SP25}, with the conjecture being recently confirmed in \cite{B-LS26}, utilizing the key diffeomorphism of $PSL_2(\C)\setminus PSU(2)$ and $(0,\infty)\times\mathcal{S},$ where $\mathcal{S}$ is the circle bundle over the projective quadric surface $Q_2(\C)\subset\C\p^3=\p_{\C}(\operatorname{Mat}_{2\times 2}\C)$ given by $\{det=0\}$. A fiber of $\mathcal{S}$ over the point $[B]_{\C^*}\in Q_2(\C)$ is given by $$\{[cB]_{\R^*}|c\in U(1)\subset {\C}^*\}.$$
Concretely, the diffeomorphism is 
\begin{align*}
	\psi: (0,\infty)\times\mathcal{S}&\longrightarrow PSL_2(\C)\setminus PSU(2)\\
	(\alpha,[B]_{\R^*})&\longmapsto [e^\alpha B+e^{-\alpha}(B^*)^{adj}]_{\C^*},
\end{align*}

Thus, under this identification, it is convenient to describe the images of surfaces, we will do it in what can be referred to as the “primitive” case and be treating the old and even degrees separately.
\subsection{Primitive surface diagrams of even degree}

Let $d=2k$ be the even degree, and $0<\gamma_1<... <\gamma_k<+\infty$ to be the “critical levels”. Let $f_0,f_2,...,f_{2k}$ be complex polynomials in the entries of $2\times2$ matrices, such that $f_j$ is homogeneous of degree $j$. Denote by $\overline{f_j}$ the restriction of this polynomial to the quadric surface $Q_2(\C)$ - it is a bi-homogeneous polynomial of bi-degree $(j, j)$ defining a curve $C_j\subset Q_2(\C)$. In addition to having the maximal number of critical levels, which is $k =\dfrac{d}{2}$, in our situations, the primitivity assumption consists of smoothness of all $C_j$, as well as transversality of $C_j$ and  $C_{j+2}$. The patchworking polynomial with this data is defined as $F_d(B)=\displaystyle\sum_{j=0}^{k}(det(B))^{2k-2j}f_j(B)t^{\beta_j}$, where $\beta_j$'s are such that the classical one variable tropical polynomial of degree $d$ $$T\longrightarrow \displaystyle\max_{j=0,...,k} (\beta_j + 2jT)$$ has the roots  at $\gamma_1,...,\gamma_k$, each with multiplicity 2. Then, the surface $S_d\subset PSL_2(\K)$ given by $F_d=0$, has the following image under $\operatorname{VAL}$:
$$\bigg(\bigsqcup_{j=1}^k(\gamma_j,\gamma_{j+1})\times \mathcal{S}|_{C_j}\bigg)\bigsqcup\bigg(\bigsqcup_{j=1}^k\{\gamma_j\}\times Im(\sigma_j)\bigg)$$
where $\gamma_{k+1}=+\infty$, and $\mathcal{S}|_{C_j}$ is the total space of the restriction to $C_j$ of the circle bundle $\mathcal{S}$ and $Im(\sigma_j)$ is the image of a section $\sigma_j$ of $\mathcal{S}$ defined away from $Cj\cup C_{j+1}$ as 
$$\sigma_j[B]_{\C^*}=\displaystyle\bigg[\sqrt{\dfrac{f_j(B)}{f_{j+2}(B)}}B\bigg]_{\R^*}.$$ The closure in $\C\p^3$ of $\operatorname{VAL}(S_d)$ is called a primitive even degree $d=2k$, $PSL_2(\C)$-phase tropical diagram. It is not hard to see, that this diagram is a topological $4$-manifold. It was conjectured by Ilia Zharkov that this manifold is homeomorphic to a degree $d$ smooth complex surface in $\C\p^3$.

\subsection{Primitive surface diagrams of odd degree}
The main structural difference of the odd degree is that
now all levels are occupied, with the contribution at $\{0\}\times PSU(2)$ being some spherical coamoeba, i.e. the image of an algebraic surface in $PSL_2(\C)$ under the projection $$\varkappa^o: PSL_2(\C)\longrightarrow PSU (2)$$ given by taking the unitary part in the polar decomposition: $$\varkappa^o[A]_{\C^*} = [A + (A^*)^{adj}]_{\R^*} \in PSU(2).$$ In the primitive case, the requirement is that this coamoeba is that of a plane. Such coamoebas can be fully described, with the generic one being the exterior of some quadratic cone in $PSU(2)\simeq \RP^3$:

This contribution complicates the situation, it is not so evident that the junction of the closure of this coameoba with $(0, \gamma_1) \times C_1$ gives a manifold. Nevertheless, 
as communicated by Ilia Zharkov, that it is plausible, and the overall topology of the odd degree $PSL_2$ diagram is again that of a complex surface. In application to the real case, it seems, one needs to provide a special care for the above coamoeba, especially since we are dealing with three different real structures.

\section{Real Structures on $PGL_2(\C)$}
The general idea of our construction is that we will be fixing some real structure, i.e. an anti-holomorphic involution, on $PGL_2(\C)$, that is $I: PGL_2(\C) \longrightarrow PGL_2(\C)$, and be looking at the $I$-real points of the complex phase diagram of the previous section $D$ such that the involved homogeneous polynomials $f_j$ are invariant with respect to $I$, i.e. $f_j\circ I=f_j$, i.e. at $D^I=\{p\in D|  I(p)=p\}$.

The main conjecture we are putting forward, which will be fully proven in the sequel article, is that provided $D$ was primitive, $\overline{D^I}\subset\overline{PGL_2(\R)}\simeq \RP^3$ is realized as an isotopy type of a smooth real algebraic surface of corresponding degree (in addition, one may note that the real part of the complex surface degenerating to $D$ for large t, is delivering precisely an example of such a surface). To this end, for the preferred list of three real structures that will be thoroughly discussed below, we've been able to verify the above conjecture up to degree $3$. Moreover, what we see that all the possible isotopy types of real algebraic surfaces are realizable by our construction. The situation in degree $4$, where the complete classification
is known classically is the subject of our current investigation. We expect to see most, if not all, of the types to be realizable by the present technique. In addition, one may hope, provided that the main conjecture holds, that some previously
unknown types of quintic surfaces will be produced.

\subsection{Real structure $I_{T^2}$}
One obvious instance of a real structure on $PGL_2(\C)$ is given by basic complex conjugation of each entry of a matrix:
$$I_{T^2}\begin{bmatrix}
	a & b\\
	c & d
\end{bmatrix}_{\C^*}=
\begin{bmatrix}
	\overline{a} & \overline{b}\\
	\overline{c}& \overline{d}
\end{bmatrix}_{\C^*}.$$

The subindex $T^2=S^1\times S^1$ refers to the topology of the real locus of the quadric $Q_2(\C)\simeq \C\p^1\times \C\p^1$ which in this situation is simply $\R\p^1\times \R\p^1$; a one sheeted hyperboloid homeomorphic to the two-dimensional topological torus $T^2$.

This real structure is, in fact, a group automorphism of $PGL_2(\C)$, thus its fixed
point set must carry a group structure. Of course, this group is, quite evidently, $PGL_2(\R).$
It is not connected topologically, with the component group being isomorphic to $\{\pm 1\}$ which corresponds to the sign of the determinant (note that for projectivized 2-by-2 real matrices this sign is well-defined). The maximal connected subgroup of $PGL_2(\R)$  is $PSL_2(\R).$ The topology of the latter is easy to describe: $PSL_2(\R)\simeq S^1\times D^2$,  where $S^1$ arises as a maximal compact subgroup $PSO(2)$, and $$D^2\simeq \mathbb{H}^2 \simeq {PSL_2(\R)}\slash{PSO(2)}$$ carries the natural structure of a hyperbolic
plane. In other words, $PSL_2(\R)$ is seen as an open solid torus topologically, and $PGL_2(\R)$ is the union of two such tori.

We compactify $PGL_2(\R)$ to $\R\p^3$ -- the projectivization of the space of real two-by-two matrices is thus decomposed into the union of two solid
tori joining at the real part, with respect to $I_{T^2}$, of the quadric surface, i.e. the torus. In this way, we have recovered the familiar genus-1 Heegaard  decomposition of the real projective 3-space thought of as a topological 3-dimensional manifold. 

Recall that the datum of genus $g$ Heegaard decomposition consist of an element in the mapping class group of a genus $g$ surface which is used to glue two copies of a solid genus $g$ handle-bodies along their boundary. In the situation above, $g=1,$ and the mapping class group is $SL_2(\mathbb{Z}).$ The mapping class group element can be found from the basic geometry of the projective $3$-space. Namely, the boundary torus of the decomposition seen as a hyperboloid intersects a real plane $\mathbb{R}\p^2$ either in a curve isotopic to a diagonal or an anti-diagonal, if we fix an identification of the hyperboloid with $\mathbb{R}\p^1\times\mathbb{R}\p^1.$ In both cases, this curve is a non-empty smooth conic on the projective plane, and thus cuts it into a disc and a M\"obius band, which switch the roles if we exchange diagonal and anti-diagonal. Thus, what we know is that $(1,1)$ curve maps to $(0,1)$ and $(1,2)$ curves on the boundary of the two solid tori. Combining this information with a similar assignment for the anti-diagonal, and taking into account that $(1,1)$ and $(-1,1)$ form a basis of the rational first homology of the torus, we get the mapping class group element.

%One may draw this, rather schematically as follows: where the $T^2$ is represented by the square with the opposite sides being glued, and the lines at the left and right extremities correspond to the non-trivial core cycles of the corresponding solid tori, i.e. $S^1\times\{0\}\subset S^1\times D^2$, where $0\in D^2$ is the center of the disc.

\subsection{Real structure $I_{\emptyset}$}
Another real structure, which is again a group automorphism of $PGL_2(\R)$, is given by
$I_{\emptyset}[A]_{\C^*}=[(A^*)^{-1}]_{\C^*}.$ Its real locus is $PSU(2)$, which doesn't intersect the quadric $Q_2(\C)$ at all, thus the index $\emptyset$ corresponding to the real locus of the quadric surface being empty.

This real structure is almost useless for our purposes: in the case of even degree
primitive diagram $D\in PGL_2(\C)$, its real locus with respect to $I_\emptyset$ is empty, as there is no spherical coamoeba contribution; and in the odd degree case we would see just the coamoeba of a $I_\emptyset$-real plane,
which is a plane in $\RP^3\simeq PSU(2)$.

\subsection{Real structure $I_{S^2}$}
The third real structure we will be concerned with is $I_{S^2}: PGL_2(\C) \longrightarrow PGL_2(\C)$ given by $I_{S^2}[A]_{\C^*} = [A^*]_{\C^*}$. Again, the index $S^2$ corresponds to the topology of the real locus of the quadric surface $Q_2(\C)$ with respect to this real structure. Importantly, $I_{S^2}$ is not a group automorphism of $PGL_2(\C)$ (but an anti-automorphism), so its real locus cannot be expected to be a subgroup. Quite evidently, this real locus, after compactification, is $\RP^3$ realized as the projectivization of the space of Hermitian 2-by-2 matrices. The real locus of the quadric, i.e. $S^2$ cuts this projective space into two non-homeomorphic connected components, one being an open 3-ball, which may be identified with $\mathbb{H}^3$, and with the other component being non-contractible. Matrix-wise, the components correspond to definite and indefinite non-degenerate Hermitian matrices -- again, these two cases are distinguidhed by the sign of the determinant. We may represent this decomposition as in Figure \ref{fig:I_{S^2}-real}.

\begin{figure}[H]
    \centering
    \includegraphics[width=0.8\linewidth]{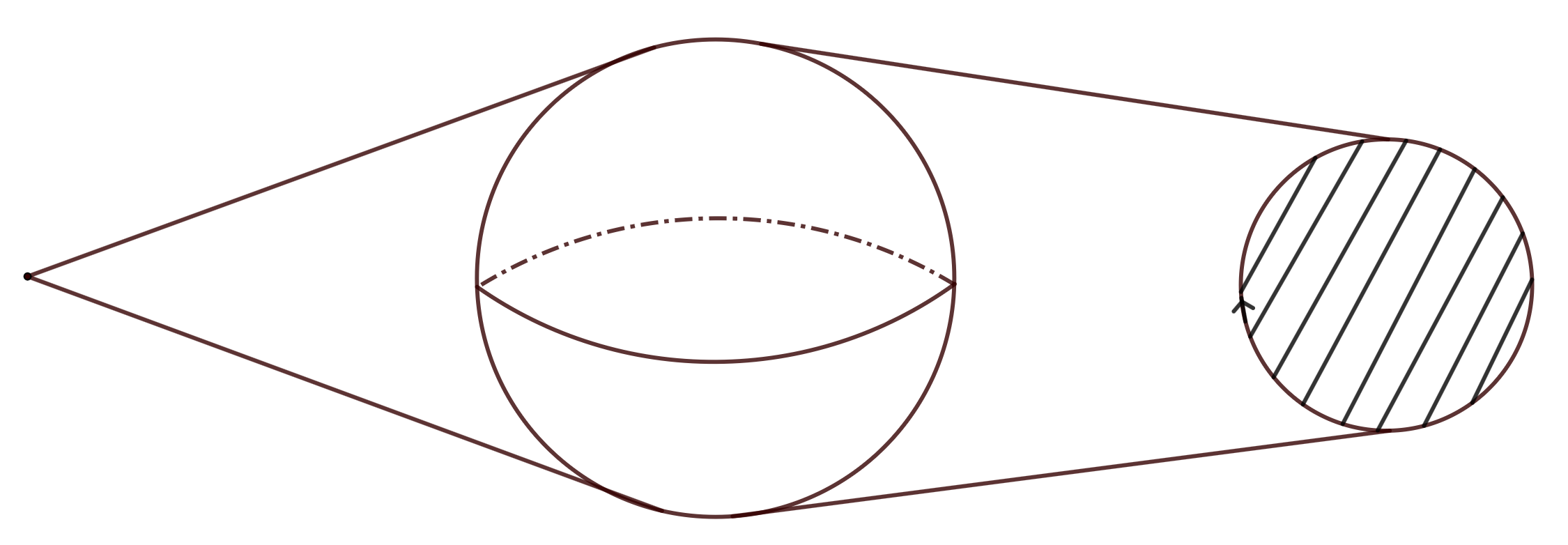}
    \caption{A schematic representation of the real locus of $PGL_2(\C)$ with respect to $I_{S^2}.$}
    \label{fig:I_{S^2}-real}
\end{figure}
Here, the right “cone” ends with $\RP^2$, i.e. the indefinite matrices are seen  as $$\bigslant{[0,1]\times S^2}{(1,p)\sim(1,-p)},$$
and the left cone ends at a point, the base, i.e. the level in the middle represents $S^2.$

\section{Topology of real $PGL_2(\C)$-phase diagrams}

In this section, we study the topology of $D^I$, for $I\in\{I_{T^2},I_\emptyset,I_{S^2}\}$.  For this purpose, we need to consider separately the real loci arising from each of the real structures introduced in the previous section.

We begin by recalling the initial data. Note that the real part of $\mathcal{S}$ is a real $S^0-$bundle over $Q(\mathbb{R})$; in other words, the fiber over a point $[B]_{\C^*}$ consists of two real points, namely $[B]_{\R^*}$ and $[iB]_{\R^*}$.

For each critical level, together with a choice of bihomogeneous polynomials $f_j$ and $f_{j+2}$, the section $\sigma_j$ uniquely determines a sign on each connected component of the exterior of $C_j\cup C_{j+2}$ at that critical level depending on the sign of the ratio $f_j/f_{j+2}$, which is a well-defined notion on $\RP^1\times \RP^1$. After taking the closure, this produces an additional critical level obtained by interchanging the preassigned signs inside the other connected component of $PGL_2(\mathbb{R})$ in a way that $\displaystyle\bigg[\sqrt{\dfrac{f_j(B)}{f_{j+2}(B)}}B\bigg]_{\R^*}$ corresponds.

\begin{figure}[H]
    \centering
    \includegraphics[width=0.9\linewidth]{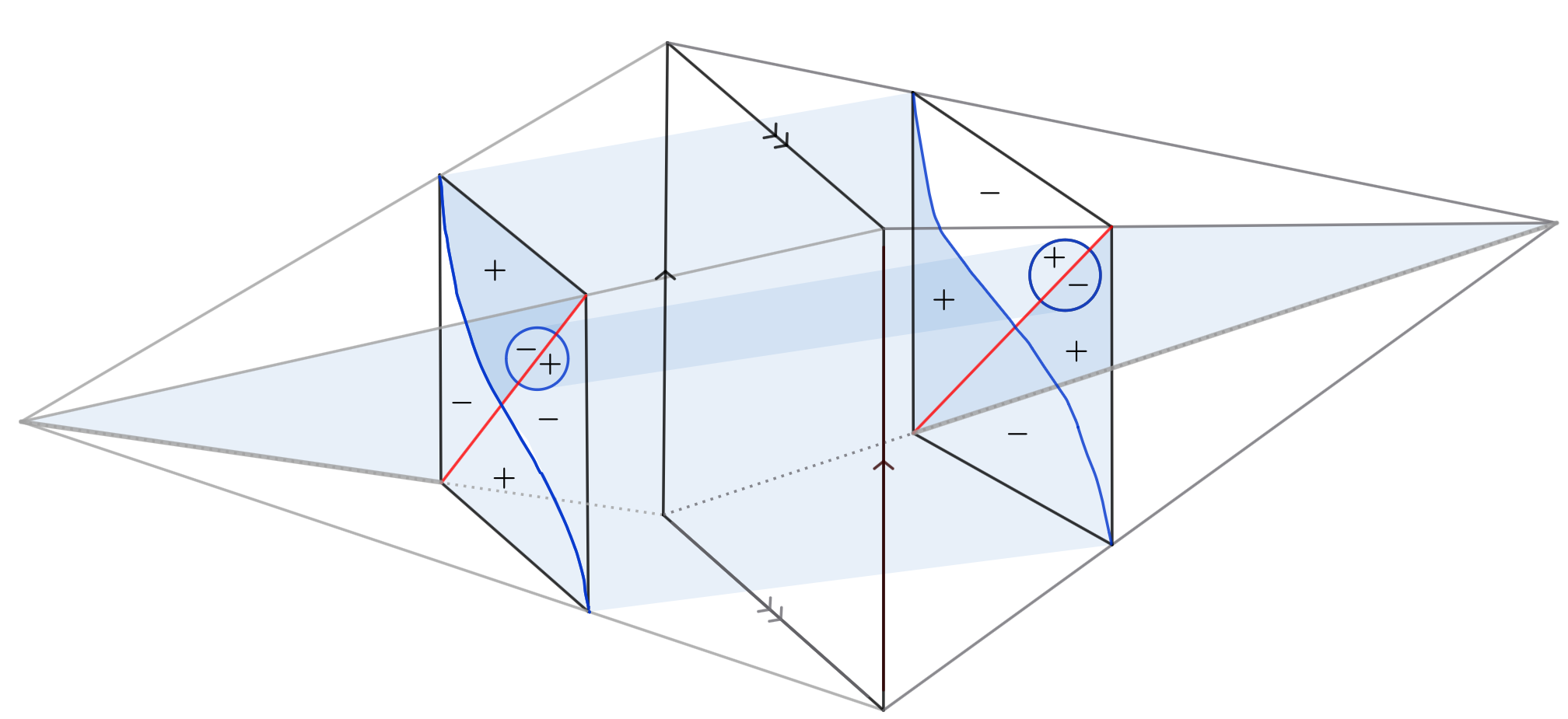}
    \caption{This diagram represents $D^{I_{T^2}}$ for degree 3 in a way that both ends are shown as points rather than a point and an $S^1$. In this particular example, we have a $(1,1)$ curve (red) and a (3,3) curve (blue) at the critical level, intersecting at four points, together with a choice of signs on each connected component of the complement of the union of the curves.}
    \label{fig:d=3_example}
\end{figure}

\begin{theorem}\label{theorem_ec}
	Let $X\subset\C\p^3$ be a complex smooth surface of degree $d$ and $D$ be the $PGL_2(\mathbb{C})$-phase diagram of degree $d$, obtained with the initial data described above. If $D$ is ${I_{T^2}}$-real and $d=2k+1,$ then the Euler characteristic is bound to
	$$\chi(\overline{D^{I_{T^2}}})\in[1,\sigma(X)]\cap (2\mathbb{Z}+1)$$
	%(-8k^3-12k^2+2k+3)/3
	In the even case and $D$ being real ${I_{T^2}}$-real, where $d=2k$,
	$$\chi(\overline{D^{I_{T^2}}})\in[0,\sigma(X)]\cap 2\mathbb{Z}.$$
	%(-8k^3-12k^2+2k+3)/
	Moreover, for $I_\emptyset$-real $D$ \[  \chi(\overline{D^{I_{\emptyset}}})=\left\{
	\begin{array}{ll}
		0, & \text{if } d \text{ is even} \\
		1, & \text{if } d \text{ is odd.} \\
	\end{array} 
	\right. \]
	Lastly, for $I_{S^2}$-real $D$ $$d\geq\chi(\overline{D^{I_{S^2}}})\geq d+\sigma(X)\text{ with }\chi(\overline{D^{I_{S^2}}})\equiv d \quad (\text{mod 2}).$$
\end{theorem}
\begin{proof}
	Let us start with the calculations for $\chi(\overline{D^{I_{T^2}}})$ in the even degree case. We first observe that between any two consecutive critical levels $\gamma_j$ and $\gamma_{j+1}$ the real locus consists of cylindrical pieces whose base is the real smooth curve 
	\[
	C_{j+1}(\mathbb{R}) \subset \mathbb{RP}^1 \times \mathbb{RP}^1
	\]
	of bidegree $(2j+2,2j+2)$. Since these pieces are topologically cylinders, they do not contribute to the Euler characteristic. Therefore, for the computation of $\chi$, the essential contribution comes from the behavior at the critical levels themselves, together with the additional coamoeba contributions in the odd degree case.
	
	The general principle determining the computation at a critical level $\gamma_{j+1}$ is the following: the union of a critical level with its reflection in the other connected component of $PGL_2(\mathbb{R})$ reconstructs the real quadric (with the intersection of the curves counted twice)
	
	$$Q(\mathbb{R}) \cong \mathbb{RP}^1 \times \mathbb{RP}^1.$$
	
	From this we obtain, for $j>1$, a contribution of the form
	
	$$\chi\!\left(Q(\mathbb{R})\right)
	+\chi\!\left(C_j(\mathbb{R})\right)
	+\chi\!\left(C_{j+1}(\mathbb{R})\right)
	-\chi\!\left(C_j(\mathbb{R}) \cap C_{j+1}(\mathbb{R})\right).$$
	
	The case $j=1$ is exceptional: at the first critical level only the curve $C_1$ appears, and hence no additional contribution arises from reflection.
	Therefore, we get 
	
	$$\chi(X)=\sum_{j=0}^{k-1}\big(\chi(\mathbb{RP}^1\times\mathbb{RP}^1)+\chi(C_j)+\chi(C_{j+1})-\chi(C_j\cap C_{j+1})\big),$$
	where $C_j$ is a bidegree $(2j,2j)$ smooth curve in $\mathbb{RP}^1\times\mathbb{RP}^1$ for each $j>0$, and $C_0$ is the empty set. We may write $0\leq\chi(C_l\cap C_{l+1})\leq 2(2j)(2j+2)$. This leaves 
	$$0\geq \chi(X)\geq -\sum_{l=0}^{k-1}2(2j)(2j+2)=\dfrac{-8k^3+8k}{3}=\dfrac{-d^3+4d}{3}.$$	
	Notice that the lower bound is nothing but the claimed signature $\sigma(X)$ and together with the transversality condition, we may say that $\chi(C_l\cap C_{l+1})$ is always even. Therefore, we have the necessary count. 
	
	In the case  $d=2k+1$, a similar computation applies. The only difference is that one must additionally include the contributions coming from the coamoeba, namely one circle and one isolated point:
	$$\chi(X)=\chi(S^1)+\chi({\text\{point\}})+\sum_{l=0}^{k-1}\big(\chi(\mathbb{RP}^1\times\mathbb{RP}^1)+\chi(C_l)+\chi(C_{l+1})-\chi(C_l\cap C_{l+1})\big),$$
	where $C_l$ is a bidegree $(2l+1,2l+1)$ smooth curve in $\mathbb{RP}^1\times\mathbb{RP}^1$ for each $l$. We may write $0\leq\chi(C_l\cap C_{l+1})\leq 2(2l+1)(2l+3)$. This leaves 
	\begin{align*}
		1 \;\geq\; \chi(X)\geq\;
		1-\sum_{l=0}^{k-1} 2(2l+1)(2l+3)=\frac{-8k^3-12k^2+2k+3}{3}&=-\frac{(2k-1)(2k+1)(2k+3)}{3} \\
		&=-\frac{(d-2)d(d+2)}{3} \\
		&=-\frac{d(d^2-4)}{3}.
	\end{align*}	A similar argument as in the first part shows that we have the required formula again.
	
	The computation of $\chi(\overline{D^{I_{\emptyset}}})$ follows from the fact that $\overline{D^{I_{\emptyset}}}=\emptyset$ in the even degree, while coamoeba effect contributes only a single real plane in the odd degree case.
	
	For the last inequality $d\geq\chi(\overline{D^{I_{S^2}}})\geq d-\sigma(X)$, we may compare it with the previous computations of $\chi(\overline{D^{I_{S^2}}})$.
	
	At each critical level, the contributions of $\chi(\RP^1\times\RP^1)$ are replaced by non-trivial contributions of $\chi(S^2)$. Since $\chi(S^2)=2,$ this produces an additional  $2k$ contribution to $\chi(\overline{D^{I_{S^2}}})$. Moreover, in the odd case, a circle is replaced by a copy of $\RP^2$, contributing an additional  $\chi(\RP^2)=1$. Therefore, in both cases  $d=2k$ or $d=2k+1$, the total additional contribution equals $d$. This establishes the claimed inequalities.
	
\end{proof}

As a remark, we add that all the above values of the Euler characteristic are realizable, this, among other things, will be explained in detail in our next article. 

\begin{theorem}
	$\overline{D^{I}}$ is a topological surface for $I\in\{I_{T^2},I_\emptyset,I_{S^2}\}$ and $D$ being $I$-real.
\end{theorem}
\begin{proof}
For the real structure $I_{\emptyset}$, there is nothing to prove by the above discussion. For the other cases, since the regions between the critical levels are cylindrical, it suffices to analyze what happens at the critical levels, and in the case where $d$ is odd, also at the two ends. In the latter case, we may have a point or $S^1$ or $\RP^2$ but in each of these have cylindrical neighborhoods of those in $\overline{D^I}$.

Let us take a real point $[B]_{\R^*}\in\overline{D^I}$ from a critical level $\gamma_l$. This element lies in one of the connected components of the critical level. If it does not belong to the boundary, which is contained by $C_l\cup C_{l+1}$, then it is an interior point of the connected component. In this case, there exists a neighborhood of $[B]_{\R^*}$ that is entirely contained in $D^{I}$.

Hence, the only nontrivial situation occurs when $[B]_{\C^*}$ lies on the boundary. We may further assume that $[B]_{\C^*}$ is a transverse intersection point of the two curves $C_l$ and $C_{l+1}$, since the remaining boundary cases are relatively straightforward.

Now consider a path $\alpha_1$ starting from a point on $C_l$ to a point on $C_{l+1}$ in such a way that it locally covers a quarter disk around $[B]_{\mathbb{R}^*}$ in $\overline{D^I}$. From its endpoint, we may choose another path $\alpha_2$ around $[B]_{\mathbb{R}^*}$ on the cylinder, that is, on the real locus of $(\gamma_j,\gamma_{j+1}) \times \mathcal{S}\big|_{C_{j+1}}$, ending at a point on $C_{l+1}$ at the critical level $\gamma_j$.

Similarly, at the same critical level, we can obtain another quarter disk at a point opposite to the starting point with respect to $[B]_{\R^*}$ by using the path $\alpha_3$. Finally, through a path $\alpha_4$ on the cylinder between the critical levels $\gamma_{j-1}$ and $\gamma_j$ that surrounds a half disk, we can return to our starting point. In this case, the region bounded by $\alpha_1 \cup \alpha_2 \cup \alpha_3 \cup \alpha_4$ in $\overline{D^I}$ gives a Euclidean neighborhood of the point $[B]_{\R^*}$.
\end{proof}
To determine whether $\overline{D^{I_{T^2}}}$ is connected, it is sufficient to understand what happens at the critical levels. In the case $d = 3$, there is only one critical level $\gamma_1$. Suppose we take two points $[B_1]_{\R^*}$ and $[B_2]_{\R^*}$ lying in two different connected components at the critical level $\gamma_1$. Starting from the connected component containing $[B_1]_{\R^*}$, we pass through the cylinder over a branch, say $C_{3,1}$ of $C_3$ on the boundary of that component, to the reflection level, and from there we use a connected component at that level to reach the boundary of the component containing $[iB_2]_{\R^*}$. Lastly, we return to $\gamma_1$ in a similar way and construct a path to $[B_2]_{\R^*}$. 
\begin{remark}\label{remark_connectT^2}
$\overline{D^{I_{T^2}}}$ is connected for $d=3$.
\end{remark}
If we consider curves $C_1$ and $C_3$ with empty $I_{S^2}$-real loci to construct $\overline{D^{I_{S^2}}}$, we have $S^2$ at $\gamma_1$, $\emptyset$ at its reflection, and $\RP^2$ at the right end of the cone. Therefore, we have the following remark.
\begin{remark}\label{remark_ec}
 For $d=3$,   $\overline{D^{I_{S^2}}}=S^2\sqcup\RP^2$ with $\chi(\overline{D^{I_{S^2}}})=3.$
\end{remark}
Let us check the topological surfaces that can be constructed by our diagrams for small degrees. With empty initial real curves, $\overline{D^{I_{S^2}}}$  is basically  an $\RP^2$ for $d=1$ and a connected space with $\chi=2$ for $d=2$. $\overline{D^{I_{T^2}}}$ is an $\RP^1\times\RP^1$ for $d=2$

The Euler characteristic $\chi(\overline{D^{I_{T^2}}})$ can take any value from the set $\{1,-1,-3,-5\}$ mentioned in Theorem \ref{theorem_ec}, since it is easy to construct smooth bi-degree $(1,1)$ and $(3,3)$ curves in $\RP^1\times \RP^1$ having 2,3,4 or 6 transverse intersection points. Combining these with Remarks \ref{remark_connectT^2} and \ref{remark_ec}, we have the following result:
\begin{remark}
All isotropy types of smooth surfaces in $\RP^3$ up to degree 3 are realizable by $PGL_2$ patchworking. \end{remark}

\section{Discussion}
We would like to first address perhaps the most immediate questions that the reader might have at this point, namely on the status of our construction and why we had put it forward at this stage. On the one hand, we are quite confident that the main
conjecture is true, i.e. everything produced by the primitive $PGL_2$ patchworking  indeed corresponds to an isotopy type of a smooth real algebraic surface in $\RP^3$. One reason for its credibility is a direct verification in small degrees which can easily be replicated, especially since we made a particular effort to make the construction transparent and ready to apply for anyone working in such fields as real algebraic geometry or low-dimensional and geometric topology. Although we have several potential schemes of how to derive the conjecture, the
complete proof is still missing, and it might take a considerable time
for it to manifest, and, from the more reserved cautious perspective, it might
turn out that the statement is, after all, false, which may come out by providing
a concrete diagram degree $d$ primitive PGL2 diagram not corresponding to any istotopy type of a real algebraic surface of degree $d$, we therefore invite the reader to try doing so.

Moreover, such  efforts at the early stage of development of this branch of the subject may lead to finding previously unknown isotopy types (recall that starting from degree five we do not even know what is the possible number of connected components of a surface), and provide improved asymptotic lower bounds on individual Betti numbers for large degrees. One of the reasons in favor of the feasibility of such attempts is due to an apparently very different nature of our construction compared to the iconic Viro's combinatorial patchworking -- in what we offer, there is practically no combinatorics involved, and, instead we are facing the dimension reduction to the problem of relative positions of smooth transversally intersecting curves on real quadric surfaces, which by no means being easy in general, seems still to be more accessible than spatial geometry of surfaces.

We tried to convince the reader that our construction has some other advantages, namely the Euler characteristic of the primitive $PGL_2$ patchworking may vary for the fixed degree, which is a stark contrast with Viro's primitive patchworking being constrained by Itenberg's theorem stating that its Euler characteristic is always equal to the signature of the corresponding complex surface. A fair question with respect to that is whether we are not shifting the terminology to produce this
distinction artificially. The answer to it relies on understanding what is the geometric meaning of the combinatorial primitivity condition: if one is given a classical tropical hypersurface, then its dual subdivision of the Newton polyhedron is primitive if and only if for every family of complex hypersurfaces, its generic member is smooth, and, moreover, is homeomorphic to the phase tropical hypersurface fibered over the tropical hypersurface that we started with, as theorems of Kerr and Zharkov, and of Kim and Nisse suggest. In our situation, the notion of primitivity is derived from the extrapolation of this principle to the $PGL_2$ case: we expect it to eventually be described by the general motto “primitive phase tropicalization restores the topological type.”

\section*{Acknowledgements}

Turgay Akyar acknowledges ICMS-Sofia and Ernesto Lupercio for providing a supportive research environment. Mikhail Shkolnikov would like to thank Grigory Mikhalkin, Peter Petrov, Ilia Zharkov and Andrei Bengu\textcommabelow s-Lasnier for invaluable discussions.

This work is supported by the Simons Foundation, grant [SFI-MPS-T-Institutes-00007697],
the Ministry of Education and Science of the Republic of Bulgaria,
grant [DO1-239/10.12.2024], as well as by the National Science Fund,
The Ministry of Education and Science of the Republic of Bulgaria, under contract [KP-06-N92/2] for the second author.

{\centering
Institute of Mathematics and Informatics\\
Bulgarian Academy of Sciences\\
Akad. G. Bonchev, Sofia 1113, Bulgaria \\\vspace{5pt} 

Department Of Mathematics, \\Middle East Technical University, \\06800 Ankara, Türkiye\\\vspace{5pt}

t.akyar[at]math.bas.bg\\\vspace{5pt}

Institute of Mathematics and Informatics\\
Bulgarian Academy of Sciences\\
Akad. G. Bonchev, Sofia 1113, Bulgaria \\\vspace{5pt}

m.shkolnikov[at]math.bas.bg\\ 
}
\end{document}